\documentclass
{amsart}

\usepackage[dvipdfmx]{graphicx}
\usepackage{tikz}
\usepackage{amssymb,amsmath,amsthm,amscd,amsfonts}

\newtheorem{theorem}{Theorem}[section]
\newtheorem{proposition}[theorem]{Proposition}
\newtheorem{corollary}[theorem]{Corollary}

\theoremstyle{definition}
\newtheorem{definition}[theorem]{Definition}
\newtheorem{remark}[theorem]{Remark}

\newtheorem{problem}[theorem]{Problem}

\begin{document}


\title[]
{Gauss maps of the Ricci-mean curvature flow}


\author{Naoyuki Koike}
\address{Department of Mathematics, Faculty of Science, Tokyo University of Science, 1-3 Kagurazaka, Shinjuku-ku, Tokyo 162-8601, Japan}
\email{koike@rs.kagu.tus.ac.jp}

\author{Hikaru Yamamoto}
\address{Department of Mathematics, Faculty of Science, Tokyo University of Science, 1-3 Kagurazaka, Shinjuku-ku, Tokyo 162-8601, Japan}
\email{hyamamoto@rs.tus.ac.jp}




\begin{abstract}
In this paper, we investigate the Gauss maps of a Ricci-mean curvature flow. 
A Ricci-mean curvature flow is a coupled equation of a mean curvature flow and a Ricci flow on the ambient manifold. 
Ruh and Vilms \cite{RuhVilms} proved that the Gauss map of a minimal submanifold in a Euclidean space is a harmonic map, 
and  Wang \cite{Wang} extended this result to a mean curvature flow in a Euclidean space by proving 
its Gauss maps satisfy the harmonic map heat flow equation. 
In this paper, we deduce the evolution equation for the Gauss maps of a Ricci-mean curvature flow, 
and as a direct corollary we prove that the Gauss maps of a Ricci-mean curvature flow satisfy the vertically harmonic map heat flow equation 
when the codimension of submanifolds is 1. 
\end{abstract} 






\thanks{The first author was supported by JSPS KAKENHI Grant Number 25400076. The second author was supported by JSPS KAKENHI Grant Number 16H07229.}



\maketitle



\section{Introduction}\label{intro}
Let $F:M\to \mathbb{R}^{n}$ be an immersion from an $\ell$-dimensional manifold $M$. 
For each point $p$ in $M$, the tangent space $F_{*}(T_{p}M)$ of $F(M)$ at $p$ 
is an affine subspace through $F(p)$, and translate it so that it passes through the origin on $\mathbb{R}^{n}$. 
Then, we get an $\ell$-dimensional (linear) subspace in $\mathbb{R}^{n}$, that is, an element of a Grassmannian $G_{\ell}(\mathbb{R}^{n})$, and denote it by $\gamma_{F}(p)$. 
The map $\gamma_{F}:M\to G_{\ell}(\mathbb{R}^{n})$ is called the Gauss map, and it has been appeared and studied in many situations, especially geometry of submanifolds. 
One of famous results about the Gauss map is due to Ruh and Vilms \cite{RuhVilms}. 
They proved that 
\begin{align}\label{RV}
\tau(\gamma_{F})=\nabla^{N}H(F), 
\end{align}
where $\tau(\gamma_{F})$ denotes the tension field of the Gauss map $\gamma_{F}:M\to G_{\ell}(\mathbb{R}^{n})$, 
$H(F)$ denotes the mean curvature vector field of $F:M\to \mathbb{R}^{n}$ 
and $\nabla^{N}H(F)$ denotes its covariant derivative with respect to the normal connection $\nabla^{N}$ of the normal bundle over $M$. 
As a direct corollary of the equation (\ref{RV}), it follows that the Gauss map of a minimal immersion is a harmonic map. 
Thus, we have a bridge from minimal submanifolds in Euclidean spaces to harmonic maps into Grassmannians. 

In this paper, we extend the result due to Ruh and Vilms from the viewpoint that it is just a static aspect of some more general phenomena which stand for flows of submanifolds. 
In this stream, Wang \cite{Wang} considered a mean curvature flow $F_{t}: M\to \mathbb{R}^{n}$, that is, a 1-parameter family of immersions which satisfies 
\[\frac{\partial F_{t}}{\partial t}=H(F_{t}). \]
Then, Theorem A in \cite{Wang} says the 1-parameter family of Gauss maps associated with $F_{t}$ is a harmonic map heat flow, that is, $\gamma_{F_{t}}:M\to G_{\ell}(\mathbb{R}^{n})$ satisfies
\begin{align}\label{harm}
\frac{\partial \gamma_{F_{t}}}{\partial t}=\tau(\gamma_{F_{t}}). 
\end{align}
In the ordinary sense, a harmonic map heat flow is defined for maps from a static Riemannian manifold $(M,h)$ to also a static $(N,\tilde{g})$. 
However, to get the simple formula (\ref{harm}), we need to use the time dependent metric $F_{t}^{*}(dx^2)$ on $M$. 
The metric on $G_{\ell}(\mathbb{R}^{n})$ is a fixed standard one. 
Hence, the notion of the harmonic map heat flow stated in \cite{Wang} is slightly different from the ordinary one. 
Of course, it is clear that (\ref{harm}) also implies that the Gauss map of a minimal immersion is a harmonic map. 

Motivated by Wang's work, in this paper we generalize Theorem A in \cite{Wang} to immersions in general ambient Riemannian manifolds. 
Let $(N,g)$ be an $n$-dimensional Riemannian manifold and $F:M\to N$ be an immersion from an $\ell$-dimensional manifold. 
We denote the Grassmann bundle which consists of all $m$-dimensional subspace in each $T_{p}N$ by $G_{m}(TN)$, where 
$m:=n-\ell$ is the codimension. We denote the projection by $\pi:G_{m}(TN)\to N$. 
Then there is a natural Riemannian metric $\tilde{g}$ called a {\it Sasaki metric} on $G_{m}(TN)$. 
See Definition \ref{Sasaki} for details. 

\begin{definition}\label{Gauss00}
We call a smooth map $\gamma_{F}:M\to G_{m}(TN)$ defined by 
\begin{align}\label{Gauss0}
\gamma_{F}(p):=(F_{*}(T_{p}M))^{\bot} \in G_{m}(T_{F(p)}N) \subset G_{m}(TN)
\end{align}
the {\it Gauss map} associated with $F$, where $(F_{*}(T_{p}M))^{\bot}$ denotes the orthogonal complement of  $F_{*}(T_{p}M)$, 
that is, the normal space of $F(M)$ at $p$. 
\end{definition}

In \cite{RuhVilms}, Ruh and Vilms considered the assignment $p\mapsto F_{*}(T_{p}M)$ which should be called the tangential Gauss map more precisely, 
and the map in Definition \ref{Gauss00} should be called the normal Gauss map. 
However, we just call the latter the Gauss map. 
When the ambient Riemannian metric is fixed, taking tangential or normal has same information. 
However, in this paper, we will deform Riemannian metrics on the ambient. 
In such a situation, the normal Gauss map is affected by the deformation, nevertheless the tangential Gauss map is not. 
Hence, taking the normal in Definition \ref{Gauss00} of the Gauss map is essential in this paper. 

To state main theorems, we prepare some notions. 
Let $(G,\tilde{g}_{t})$ and $(N,g_{t})$ be 1-parameter families of Riemannian manifolds defined on a time interval $[0,T)$. 
Assume that there exists a (time independent) submersion $\pi:G\to N$ such that it is a Riemannian submersion from $(G,\tilde{g}_{t})$ to $(N,g_{t})$ for each $t\in [0,T)$. 
We define a (time independent) distribution $\mathcal{V}$ on $G$ by $\mathcal{V}:=\ker \pi_{*}$ and call it the vertical distribution, 
and define a time dependent distribution $\mathcal{H}_{t}$ on $G$ by $\mathcal{H}_{t}:=\mathcal{V}^{\bot_{t}}$, where $\bot_{t}$ means the orthogonal complement with respect to $\tilde{g}_{t}$ 
and call it the horizontal distribution. 
Then, by the assumption, the restriction of $\pi_{*}$ to $\mathcal{H}_{t}$ is a fiberwise linear isometry from $(\mathcal{H}_{t},\tilde{g}_{t})$ to $(TN,g_{t})$. 
We denote the vertical part of a vector field $X$ on $G$ by $X^{v}$. 

\begin{definition}\label{vhmhf}
A pair of 1-parameter families of Riemannian manifolds $(M,h_{t})$ and smooth maps $\gamma_{t}:M\to G$ defined on a time interval $[0,T)$ 
is called a {\it vertically harmonic map heat flow} if it satisfies 
\begin{align}\label{vhf}
\left(\frac{\partial \gamma_{t}}{\partial t}\right)^{v}=\tau(\gamma_{t})^{v}, 
\end{align}
where $\tau(\gamma_{t})$ is the tension field of the map $\gamma_{t}:(M,h_{t}) \to (G,\tilde{g}_{t})$. 
\end{definition}

In this paper, $G$ is just the Grassmann bundle $G_{m}(TN)$ over $N$ and $\tilde{g}_{t}$ is the Sasaki metric associated with $g_{t}$. 
For the definition of the Sasaki metric, see Definition \ref{Sasaki}. 
In such a case, as explained in Section \ref{RGGB}, a fiber $\mathcal{V}_{W}$, which is equal to $T_{W}(G_{m}(T_{\pi(W)}N))$, over a point $W\in G_{m}(TN)$ is canonically identified with $\mathop{\mathrm{Hom}}(W,W^{\bot})$. 

\begin{definition}\label{RgF}
Let $(N,g)$ be an $n$-dimensional Riemannian manifold. 
For each point $W$ in $G_{m}(TN)$, we define an element $(\mathcal{R}(g))(W)$ in $\mathop{\mathrm{Hom}}(W,W^{\bot})\cong\mathcal{V}_{W}$ by 
\[((\mathcal{R}(g))(W))(w):=\sum_{i=1}^{m}\left(R(w,\nu_{i})\nu_{i}\right)_{W^{\bot}}, \]
where $R$ is the Riemannian curvature operator of $g$ and $\{\, \nu_{i}\,\}_{i=1}^{m}$ is an orthonormal basis of $W$ and $(\,\ast\,)_{W^{\bot}}$ means the $W^{\bot}$-component of $\ast$. 
Then, we define a vertical vector field $\mathcal{R}(g)$ on $G_{m}(TN)$, that is, a section of $\mathcal{V}$ by
\[W\mapsto (\mathcal{R}(g))(W)\]
for each point $W$ in $G_{m}(TN)$. 
\end{definition}

Then, our first main theorem is the following. 

\begin{theorem}\label{main1}
Let $(N,g_{t})$, $f_{t}:N\to \mathbb{R}$ and $F_{t}:M\to N$ be a 1-parameter family of $n$-dimensional Riemannian manifolds, 
smooth functions on $N$ and immersions from an $\ell$-dimensional manifold $M$ respectively defined on a time interval $[0,T)$ satisfying
\begin{subequations}
\begin{align}
\frac{\partial g_{t}}{\partial t}=&-\mathop{\mathrm{Ric}}(g_{t})+f_{t}g_{t} \label{rf1}\\ 
\frac{\partial F_{t}}{\partial t}=&H_{g_{t}}(F_{t}), \label{rmcf1}
\end{align}
\end{subequations}
where $\mathop{\mathrm{Ric}}(g_{t})$ is the Ricci curvature of $g_{t}$ and $H_{g_{t}}(F_{t})$ is the mean curvature vector field of $F_{t}:M\to N$ with respect to the metric $g_{t}$ on $N$. 
Then the Gauss map $\gamma_{F_{t}}:M \to G_{m}(TN)$ $(m:=n-\ell)$ satisfies 
\begin{align}\label{general}
\left(\frac{\partial \gamma_{F_{t}}}{\partial t}\right)^{v}=\tau(\gamma_{F_{t}})^{v}+\mathcal{R}(g_{t})\circ \gamma_{F_{t}}
\end{align}
with respect to the induced metric $F_{t}^{*}g_{t}$ on $M$ and the Sasaki metric $\tilde{g}_{t}$ on $G_{m}(TN)$ associated with $g_{t}$. 
\end{theorem}

If $\mathcal{R}(g_{t})$ is identically zero, then the equation (\ref{general}) is reduced to the vertically harmonic map heat flow equation (\ref{vhf}). 
The typical case is the codimension 1 case. 
When $m=1$, then $ G_{1}(TN)$ is equivalent to $\mathbb{P}(TN)$ and $\mathcal{R}(g_{t})\equiv 0$ by definition since $W$ in $\mathbb{P}(TN)$ is a 1 dimensional subspace. 
Hence, as a direct corollary from Theorem \ref{main1}, we have the following. 
 
\begin{corollary}\label{cor1}
Let $(N,g_{t})$, $f_{t}:N\to \mathbb{R}$ and $F_{t}:M\to N$ be as in Theorem \ref{main1}. 
Assume the codimension of $M$ is $1$. 
Then the Gauss map $\gamma_{F_{t}}:M \to \mathbb{P}(TN)$ is a vertically harmonic map heat flow, that is, it satisfies
\[\left(\frac{\partial \gamma_{F_{t}}}{\partial t}\right)^{v}=\tau(\gamma_{F_{t}})^{v}.\]
\end{corollary}

\begin{remark}
When $f_{t}\equiv 0$, the equation (\ref{rf1}) is equivalent to $\partial_{t} g_{t}=-\mathop{\mathrm{Ric}}(g_{t})$. 
This is very similar to the Ricci flow equation $\partial_{t} g_{t}=-2\mathop{\mathrm{Ric}}(g_{t})$ but the coefficient of $\mathop{\mathrm{Ric}}(g_{t})$ is different. 
When $f_{t}\equiv \alpha$ (constant) and $g_{t}$ is K\"ahler metric, the equation (\ref{rf1}) is equivalent to the normalized K\"ahler Ricci flow equation. 
\end{remark}

The reason why we take vertical parts in (\ref{vhf}) is the following. 
When the ambient space is $\mathbb{R}^{n}$, for an immersion $F:M^{n-m}\to\mathbb{R}^{n}$, 
the Gauss map is defined as a subspace obtained by translating a normal space at $F(p)$ so that it passes through the Origin for each point $p$ in $M$. 
Only in this paragraph, we denote it by $\bar{\gamma}_{F}:M\to G_{m}(\mathbb{R}^{n})$ to distinguish it from Definition \ref{Gauss00}. 
Then, $\gamma_{F}:M\to G_{m}(T\mathbb{R}^{n})$ by Definition \ref{Gauss00} and the ordinary definition $\bar{\gamma}_{F}:N\to G_{m}(\mathbb{R}^{n})$ is related as
\[\gamma_{F}(p)=(F(p),\bar{\gamma}_{F}(p))\]
by the natural trivialization $G_{m}(T\mathbb{R}^{n})\cong \mathbb{R}^{n}\times G_{m}(\mathbb{R}^{n})$. 
Under this notation, (\ref{RV}) and $(\ref{harm})$ are rewritten as $\tau(\bar{\gamma}_{F})=\nabla^{N}H(F)$ and $\partial_{t}\bar{\gamma}_{F_{t}}=\tau(\bar{\gamma}_{F_{t}})$. 
Namely, these are just equations in the fibers, that is, in $G_{m}(\mathbb{R}^{n})$. 
Hence, observing only the vertical component as Definition \ref{vhmhf} is natural to generalize the situations when the ambient space is $\mathbb{R}^{n}$. 
The idea taking the vertical component was also appeared in the paper of Wood \cite{Wood} in a similar situation. 
Actually, since $(\partial_{t}\gamma_{F_{t}})^{v}=\partial_{t}\bar{\gamma}_{F_{t}}$ and $\tau(\gamma_{t})^{v}=\tau(\bar{\gamma}_{t})$, we have the following. 

\begin{remark}
Taking $(N,g_{t})\equiv (\mathbb{R}^{n},dx^2)$ and $f_{t}\equiv 0$, 
Theorem \ref{main1} implies Theorem~A in \cite{Wang}, since $\mathcal{R}(dx^2)\equiv 0$. 
\end{remark}

Once we have proved that the Gauss maps satisfies the vertically harmonic map heat flow equation, 
one obtain a subsolution of some kind of parabolic equation by a similar way of Corollary A of \cite{Wang}. 
In the remainder of this introduction, we explain that. 
Let $(N,g_{t})$, $f_{t}:N\to \mathbb{R}$ and $F_{t}:M\to N$ be given as in Theorem \ref{main1}, 
and assume that $\mathop{\mathrm{codim}} M=1$. 
Then, its Gauss map $\gamma_{F_{t}}:M\to \mathbb{P}(TN)$ satisfies (\ref{vhf}). 
Let $\rho:\mathbb{P}(TN)\to\mathbb{R}$ be a smooth function which satisfies 
\begin{align}\label{horicon}
(\nabla^{t}\rho)_{\mathcal{H}_{t}}\equiv0
\end{align}
for all time $t\in[0,T)$, 
where $\nabla^{t}\rho$ is the gradient of $\rho$ with respect to the Sasaki metric $\tilde{g}_{t}$ on $\mathbb{P}(TN)$ 
and $(\nabla^{t}\rho)_{\mathcal{H}_{t}}$ is the $\mathcal{H}_{t}$-part of $\nabla^{t}\rho$. 
This condition means that $\rho$ is {\it horizontally constant} for all time. 
Assume that there exists a constant $C\geq 0$ such that 
\begin{align}\label{hessC}
\mathop{\mathrm{Hess}_{t}}\rho\geq-C\tilde{g}_{t}
\end{align}
for all $t\in[0,T)$. 
It is clear that we can take such $C$ when $N$ is compact and the flows can be extended on $[0, T+\epsilon)$ for some $\epsilon>0$. 
Then, our second main theorem is the following. 

\begin{theorem}\label{main2}
The pull-back of $\rho$ by the Gauss map, $\rho\circ \gamma_{F_{t}}:M\to\mathbb{R}$, satisfies 
\begin{align}\label{subsol}
\left(\frac{\partial}{\partial t} -\Delta_{F_{t}^{*}g_{t}}\right)(\rho\circ\gamma_{F_{t}}) \leq C\left((n-1)+|A(F_{t})|^2\right),
\end{align}
where $\Delta_{F_{t}^{*}g_{t}}$ is the Laplace operator of the induced metric $F_{t}^{*}g_{t}$ on $M$ and 
$|A(F_{t})|$ is the norm of the second fundamental form of $F_{t}:M\to N$ with respect to the metric $g_{t}$ on $N$. 
\end{theorem}

\noindent {\bf Relation to other papers.} 
The pair of equations (\ref{rf1}) and (\ref{rmcf1}) can be considered as a coupled equation of a Ricci flow (although the coefficient is different by $-2$) and a mean curvature flow. 
Recently, studies of coupled equations have been spread. 
For instance, Ricci-mean curvature flows are appeared in \cite{HanLi}, \cite{LotayPacini} and \cite{Yamamoto2}, and Ricci-harmonic map heat flows are appeared in \cite{Muller}. 
Recently, Ramos and Ripoll \cite{RamosRipoll} have proved a result similar to Ruh-Vilms' theorem 
for immersions of codimension one into symmetric spaces, where they defined the Gauss map as a map into a sphere 
in the Lie algebra of the isometry group of the symmetric space. 

\vspace{3mm}
\noindent {\bf Organization of this paper.} 
The organization of this paper is as follows. 
In Section \ref{RGGB}, we review Riemannian geometry of Grassmann bundle, define the Sssaki metric and write down its Levi-Civita connection explicitly. 
In Section \ref{TFG}, we compute the tension field of the Gauss map. 
In Section \ref{VVG}, we calculate the variational vector field of Gauss maps associated with 1-parameter families of immersions and metrics. 
In Section \ref{GCF}, we prove Theorem \ref{main1}. 
In Section \ref{Asub}, we prove Theorem \ref{main2}. 
In Appendix A, we give a natural question related to the equation (\ref{general}). 
\section{Riemannian geometry of Grassmann bundle}\label{RGGB}
A Grassmann bundle over a Riemannian manifold admits a natural Riemannian structure. 
In this section, we review some of standard facts on the Grassmann bundle as a Riemannian manifold. 
Throughout this paper, we often use expressions in this section. 

Let $(N,g)$ be an $n$-dimensional Riemannian manifold and $\pi:G_{m}(TN)\to N$ be a Grassmann bundle over $N$ for a fixed integer $0\leq m \leq n$. 
Namely, the fiber over $p$ in $N$ consists of all $m$-dimensional subspaces in $T_{p}N$. 
We introduce good coordinates on $G_{m}(TN)$ as follows. 
Fix a point $V_{0}$ in $G_{m}(TN)$ and put $p_{0}:=\pi(V_{0})$. 
Let $(U,\varphi=(x^{1},\dots,x^{n}))$ be normal coordinates of $(N,g)$ centered at $p_{0}$. 
Let $v_{1},\dots,v_{m}$ be an orthonormal basis of $V_{0}$ and $w_{m+1},\dots,w_{n}$ be an orthonormal basis of $V_{0}^{\bot}$. 
We extend $v_{i}$, $w_{j}$ on $U$ as parallel transports with respect to $\nabla$ along geodesics from $p_{0}$, and continue to denote them by $v_{i}$, $w_{j}$. 
Define a map $\Gamma:\varphi(U)\times M(m,n-m)\to G_{m}(TN)$ by assigning 
\begin{align}\label{coords}
\Gamma(x,a):=\mathop{\mathrm{Span}}\left\{\,v_{i}(\varphi^{-1}(x))+\sum_{\alpha=m+1}^{n}a_{i}^{\alpha}w_{\alpha}(\varphi^{-1}(x))\,\,\Bigg|\,\, i=1,\dots,m  \,\right\}. 
\end{align}
to each $(x=(x^{A}), a=(a_{i}^{\alpha}))$, where $M(m,n-m)$ denotes the space of all $(m,n-m)$-matrices. 
Then, $\Gamma$ gives local coordinates on $G_{m}(TN)$ defined on a sufficiently small neighborhood of $(0,O)$ in $\varphi(U)\times M(m,n-m)$ around $V_{0}$. 

First, we define the vertical distribution $\mathcal{V}$ over $G_{m}(TN)$ and a natural fiber metric $k$ on $\mathcal{V}$. 
For a point $V$ in $G_{m}(TN)$, we define the {\it vertical subspace} of $T_{V}G(TN)$ by 
the kernel of the linear map $\pi_{*V}:T_{V}G_{m}(TN)\to T_{\pi(V)}N$ and denote it by $\mathcal{V}_{V}$. 
Put $p=\pi(V)$. 
It is clear that $\mathcal{V}_{V}$ is the tangent space at $V$ of $G_{m}(T_{p}N)$, the fiber over $p$. 
Hence, there is the natural identification of $\mathcal{V}_{V}$ with $\mathop{\mathrm{Hom}}(V,V^{\bot})$ by 
\[\frac{d}{ds}\bigg|_{s=0}V(s) \quad \longmapsto \quad \sum_{i=1}^{m}v^{*}_{i}(0)\otimes \left(\frac{d}{ds}\bigg|_{s=0}v_{i}(s)\right)_{V^{\bot}}, \]
where $V(s)$ is a curve in $G_{m}(T_{p}N)$ through $V$ at $s=0$ and $v_{1}(s),\dots, v_{m}(s)$ is a basis of $V(s)$. 
Furthermore,  $(\ast)_{V^{\bot}}$ denotes the $V^{\bot}$-part of $\ast$ with respect to the orthogonal decomposition $T_{p}N=V\oplus V^{\bot}$ by the Riemannian metric $g$. 
One can easily check that this correspondence is canonical, that is, it does not depend on the second derivative of $V(s)$ and the choice of basis of $V(s)$. 
Via this identification, we define the natural inner product $k_{V}$ on $\mathcal{V}_{V}$ so that it makes 
\[\{\, v_{i}^{*}\otimes w_{\alpha}\mid 1\leq i \leq m,\, m+1\leq \alpha \leq n\,\}\]
an orthonormal basis of $\mathcal{V}_{V}\cong \mathop{\mathrm{Hom}}(V,V^{\bot})$ for an orthonormal basis $v_{1},\dots, v_{m}$ of $V$ and an orthonormal basis $w_{m+1},\dots, w_{n}$ of $V^{\bot}$. 
The {\it vertical distribution} on $G(TN)$ is defined by 
\[\mathcal{V}:=\bigcup_{V\in G_{m}(TN)}\mathcal{V}_{V}\]
and a fiber metric $k$ on $\mathcal{V}$ is defined by $k:V\mapsto k_{V}$ for each $V\in G_{m}(TN)$. 

Next, we define the horizontal distribution $\mathcal{H}$ over $G_{m}(TN)$. 
We define the horizontal lift of $u\in T_{p}N$ to $V\in \pi^{-1}(p)$ as follows. 
Take a basis $v_{1},\dots, v_{m}$ of $V$ and a curve $c:(-\epsilon,\epsilon)\to N$ through $p$ at $s=0$ so that $\frac{d}{ds}|_{s=0}c(s)=u$. 
Let $v_{i}(s)$ be the parallel transport of $v_{i}$ along $c$ with respect to the Levi--Civita connection of $(N,g)$. 
Then, we define the {\it horizontal lift} of $u$ at $V$ by 
\[[u]^{h}_{V}:=\frac{d}{ds}\bigg|_{s=0}\mathop{\mathrm{Span}}\{\, v_{1}(s),\dots,v_{m}(s)\,\}\in T_{V}G_{m}(TN). \]
We sometimes omit the subscript $V$ if it can be recovered from the context. 
One can easily see that the correspondence $T_{p}N\ni u \mapsto [u]^{h}_{V} \in T_{V}G_{m}(TN)$ is linear and injective. 
Thus, its image defines an $n$-dimensional subspace of $T_{V}G_{m}(TN)$ and we denote it by $\mathcal{H}_{V}$ and call it the {\it horizontal subspace} of $T_{V}G_{m}(TN)$. 
The {\it horizontal distribution} $\mathcal{H}$ on $G_{m}(TN)$ is defined by 
\[\mathcal{H}:=\bigcup_{V\in G_{m}(TN)}\mathcal{H}_{V}. \]

Then, the tangent bundle of $G_{m}(TN)$ is decomposed as 
\begin{align}\label{HV}
TG_{m}(TN)=\mathcal{H}\oplus\mathcal{V}. 
\end{align}
This decomposition is explicitly described as follows. 
Fix a point $V\in G_{m}(TN)$ and a tangent vector $X\in T_{V}G_{m}(TN)$, and put $p:=\pi(V)\in N$. 
Let $V:(-\epsilon,\epsilon)\to G_{m}(TN)$ be a curve through $V$ at $s=0$ so that $\frac{d}{ds}|_{s=0}V(s)=X$. 
Then the curve $c(s):=\pi(V(s))$ in $N$ defines a tangent vector $u:=\frac{d}{ds}|_{s=0}c(s)$ at $p$. 
It is clear that $u=\pi_{*}(X)$. 
Take a basis $v_{1}(s),\dots,v_{m}(s)$ of $V(s)$. 
Then, $X$ decomposes as 
\begin{align}\label{dec1}
X=[u]^{h}_{V}+\sum_{i=1}^{m}v_{i}^{*}(0)\otimes\left(\nabla_{\frac{d}{ds}\big|_{s=0}}v_{i}(s)\right)_{V^{\bot}}. 
\end{align}
The first term is the horizontal part of $X$ and the second term is the vertical part of $X$. 
We denote $\pi_{*}(X)$ by $\hat{X}$ and the vertical part of $X$ by $X^{v}$. 
Under these notations, we have $X=[\hat{X}]^{h}+X^{v}$. 
Then, a Riemannian metric on $G_{m}(TN)$ is defined so that the decomposition (\ref{HV}) becomes an orthogonal decomposition. 

\begin{definition}\label{Sasaki}
Let $(N,g)$ be an $n$-dimensional Riemannian manifold, and $\pi:G_{m}(TN)\to N$ be a Grassmann bundle over $N$. 
We define a Riemannian metric $\tilde{g}$ on $G_{m}(TN)$ by 
\[\tilde{g}(X,Y):=g(\hat{X},\hat{Y})+k(X^{v},Y^{v}), \]
and call it a {\it Sasaki metric}. 
\end{definition}

It is clear that, with respect to the Sasaki metric $\tilde{g}$, the decomposition (\ref{HV}) is orthogonal, 
and $\pi_{*}|_{\mathcal{H}_{V}}:\mathcal{H}_{V}\to T_{\pi(V)}N$ is linear isometry for each $V\in G_{m}(TN)$. 
Thus, $\pi:(G_{m}(TN),\tilde{g})\to (N,g)$ is a Riemannian submersion. 
Actually, it is well-known that each fiber is totally geodesic. 
We denote the Levi--Civita connection on $(G_{m}(TN),\tilde{g})$ by $\tilde{\nabla}$. 
The remainder of this section is devoted to see the explicit formula for $\tilde{\nabla}$. 

Here we prepare some notations. 
Denote by $\nabla$ and $R$ the Levi--Civita connection and the curvature tensor of $(N,g)$, respectively. 
Our definition of $R$ obeys 
$R(\xi_{1},\xi_{2})\xi_{3}:=\mathop{(\nabla_{\xi_{1}}\nabla_{\xi_{2}}-\nabla_{\xi_{2}}\nabla_{\xi_{1}}-\nabla_{[\xi_{1},\xi_{2}]})}\xi_{3}$ 
for tangent vector $\xi_{i}$ ($i=1,2,3$) on $N$. 
Next, for a point $V$ in $G_{m}(TN)$ and tangent vectors $\xi_{1}$, $\xi_{2}$ of $T_{\pi(V)}N$, 
an element $R^{\bot}(\xi_{1},\xi_{2})$ in $\mathop{\mathrm{Hom}}(V,V^{\bot})$ is defined by
\[(R^{\bot}(\xi_{1},\xi_{2}))(v):=(R(\xi_{1},\xi_{2})v)_{V^{\bot}}\]
for $v \in V$. 
Let $X$, $Y$ be vector fields on $G_{m}(TN)$. 
By the projection $\pi:G_{m}(TN)\to N$, we get the pull-back bundle $\pi^{*}(TN)$ over $G_{m}(TN)$, 
and we define sections of $\pi^{*}(TN)$ by 
\[\hat{X}(V):=\pi_{*}(X(V))\quad\text{and}\quad\hat{Y}(V):=\pi_{*}(Y(V))\]
for $V\in G_{m}(TN)$, respectively. 
Then, for each section $\xi$ of $\pi^{*}(TN)$, by taking the inner product of $R^{\bot}(\hat{X}(V),\xi(V))$ and $Y^{v}(V)$ 
as elements in $\mathop{\mathrm{Hom}}(V,V^{\bot})$ with respect to $k_{V}$, 
we get a function $k(R^{\bot}(\hat{X},\xi),Y^{v})$ over $G_{m}(TN)$. 
Thus, the correspondence $\xi\mapsto k(R^{\bot}(\hat{X},\xi),Y^{v})$ is a section of the dual bundle $(\pi^{*}(TN))^{*}$ of $\pi^{*}(TN)$, 
and we denote its metric dual, a section of $\pi^{*}(TN)$, by $k(R^{\bot}(\hat{X},\bullet),Y^{v})^{\flat}$. 
Finally, $\nabla^{\bot}_{X}Y^{v}$ is defined as an element in $\mathop{\mathrm{Hom}}(V,V^{\bot})$ by 
\begin{align}\label{normalnabla}
(\nabla^{\bot}_{X}Y^{v})(v_{i}(0)):=\left(\nabla_{\frac{\partial}{\partial s}|_{s=0}}(Y^{v}(v_{i}(s)))\right)_{V^{\bot}}-Y^{v}\left(\left(\nabla_{\frac{\partial}{\partial s}|_{s=0}}v_{i}(s)\right)_{V} \right), 
\end{align}
where $v_{1}(s),\dots,v_{m}(s)$ is a basis of a curve $V(s)$ in $G_{m}(TN)$ such that $V(0)=V$ and $\frac{d}{ds}|_{s=0}V(s)=X(V)$. 

\begin{proposition}\label{LC1}
The Levi--Civita connection $\tilde{\nabla}$ of the Sasaki metric $\tilde{g}$ is given by
\begin{align}\label{LC2}
\begin{aligned}
\tilde{\nabla}_{X}Y:=&\biggl[\nabla_{X}\hat{Y}+\frac{1}{2}k(R^{\bot}(\hat{X},\bullet),Y^{v})^{\flat}+\frac{1}{2}k(R^{\bot}(\hat{Y},\bullet),X^{v})^{\flat}\biggr]^{h}\\
&-\frac{1}{2}R^{\bot}(\hat{X},\hat{Y})+\nabla^{\bot}_{X}Y^{v}
\end{aligned}
\end{align}
for any vector fields $X$ and $Y$ on $G_{m}(TN)$, 
where $\nabla_{X}\hat{Y}$ means the covariant derivative of $\hat{Y}$ as a section of $\pi^{*}(TN)$ by $X\in TG_{m}(TN)$ with respect to 
the induced connection from the Levi--Civita connection $\nabla$ of $g$. 
\end{proposition}

\begin{proof}
Put
\begin{align*}
&\overline{\nabla}_{X}Y:=\biggl[\nabla_{X}\hat{Y}\biggr]^{h}+\nabla^{\bot}_{X}Y^{v}, \\
&\overline{\mathcal{R}}(X,Y):=\biggl[\frac{1}{2}k(R^{\bot}(\hat{X},\bullet),Y^{v})^{\flat}+\frac{1}{2}k(R^{\bot}(\hat{Y},\bullet),X^{v})^{\flat}\biggr]^{h}-\frac{1}{2}R^{\bot}(\hat{X},\hat{Y}). 
\end{align*}
Then, we have $\tilde{\nabla}_{X}Y=\overline{\nabla}_{X}Y+\overline{\mathcal{R}}(X,Y)$. 
It is clear that $\overline{\nabla}$ is a connection on $G_{m}(TN)$ and $\overline{\mathcal{R}}$ is a $(1,2)$-tensor field on $G_{m}(TN)$, 
and hence $\tilde{\nabla}$ is a connection on $G_{m}(TN)$. 

First, we prove that $\tilde{\nabla}\tilde{g}=0$. 
By definitions of $\tilde{g}$ and $\tilde{\nabla}$, we have 
\begin{align*}
\tilde{g}(\tilde{\nabla}_{X}Y,Y)=&g(\nabla_{X}\hat{Y},\hat{Y})+\frac{1}{2}k(R^{\bot}(\hat{X},\hat{Y}),Y^{v})+\frac{1}{2}k(R^{\bot}(\hat{Y},\hat{Y}),X^{v})^{\flat}\\
&-\frac{1}{2}k(R^{\bot}(\hat{X},\hat{Y}),Y^{v})+k(\nabla^{\bot}_{X}Y^{v},Y^{v})\\
=&g(\nabla_{X}\hat{Y},\hat{Y})+k(\nabla^{\bot}_{X}Y^{v},Y^{v}). 
\end{align*}
It is clear that $Xg(\hat{Y},\hat{Y})=2g(\nabla_{X}\hat{Y},\hat{Y})$ since $\nabla$ is the induced connection and it preserves the metric, and 
$Xk(Y^{v},Y^{v})=2k(\nabla^{\bot}_{X}Y^{v},Y^{v})$ since $\nabla^{\bot}_{X}Y^{v}$ is defined so that $k$ is parallel. 
Thus, we have proved 
\[2\tilde{g}(\tilde{\nabla}_{X}Y,Y)=X\tilde{g}(Y,Y), \]
and this means that $\tilde{\nabla}$ preserves $\tilde{g}$. 

Next, we prove that $\tilde{\nabla}$ is torsion-free. Since the horizontal part of $\overline{\mathcal{R}}(X,Y)$ is symmetric and the vertical part is skew-symmetric, we have 
\[\tilde{\nabla}_{X}Y-\tilde{\nabla}_{Y}X=\biggl[\nabla_{X}\hat{Y}-\nabla_{Y}\hat{X}\biggr]^{h}+\biggl(\nabla^{\bot}_{X}Y^{v}-\nabla^{\bot}_{Y}X^{v}\biggr)-R^{\bot}(\hat{X},\hat{Y}). \]
Since $\nabla$ is the induced connection from the Levi--Civita connection, we have
\[\nabla_{X}\hat{Y}-\nabla_{Y}\hat{X}=\widehat{[X,Y]}. \]
Thus, 
\[\biggl[\nabla_{X}\hat{Y}-\nabla_{Y}\hat{X}\biggr]^{h}=\widehat{[X,Y]}^{h}, \]
and the right hand side is just the horizontal part of $[X,Y]$. Thus, it is enough to prove 
\begin{align}\label{v1}
\biggl(\nabla^{\bot}_{X}Y^{v}-\nabla^{\bot}_{Y}X^{v}\biggr)-R^{\bot}(\hat{X},\hat{Y})=[X,Y]^{v}. 
\end{align}
If both $X$ and $Y$ are vertical, that is, these are tangent to $G_{m}(T_{p}N)$ for each $p$ in $N$, then (\ref{v1}) holds, since $\hat{X}=\hat{Y}=0$ 
and $\nabla^{\bot}$ restricted to $G_{m}(T_{p}N)$ is just the Levi--Civita connection of $G_{m}(T_{p}N)$. 
In the other cases, we prove (\ref{v1}) by using good coordinates $(x,a)\mapsto \Gamma(x,a)$ defined by (\ref{coords}) around a fixed  point $V_{0}$ in $G_{m}(TN)$. 
When
\[X=\frac{\partial}{\partial x^{A}}\quad,\quad Y=\frac{\partial}{\partial x^{B}}, \]
it is enough to prove
\begin{align}\label{v2}
\biggl(\nabla^{\bot}_{X}Y^{v}-\nabla^{\bot}_{Y}X^{v}\biggr)-R^{\bot}(\hat{X},\hat{Y})=0
\end{align}
at $V_{0}$. 
Put $V(t,s):=\Gamma(te_{A}+se_{B},0)$, where $e_{1},\dots, e_{n}$ is the standard basis of $\mathbb{R}^{n}$. 
Then we have 
\[\frac{\partial }{\partial t}\bigg|_{t=0}V(t,0)=X(V_{0})\quad\text{and}\quad \frac{\partial }{\partial s}\bigg|_{s=0}V(t,s)=Y(V(t,0)).\] 
Since 
\[v_{i}(t,s):=v_{i}(\varphi^{-1}(te_{A}+se_{B}))\] 
is a basis of $V(t,s)$, the vertical part of $Y$ is expressed as 
\[(Y^{v})(v_{i})=\left(\nabla_{\hat{Y}}v_{i}\right)_{V^{\bot}}=\sum_{\alpha=m+1}^{n}g\left(\nabla_{\hat{Y}}v_{i},w_{\alpha}\right)w_{\alpha}\]
at $V=V(t)$ by (\ref{dec1}). 
Since $\nabla_{\hat{X}}v_{i}=\nabla_{\hat{X}}w_{\alpha}=0$ at $p_{0}:=\pi(V_{0})$, we have 
\[(\nabla^{\bot}_{X}Y^{v})(v_{i})=\sum_{\alpha=m+1}^{n}g\left(\nabla_{\hat{X}}\nabla_{\hat{Y}}v_{i},w_{\alpha}\right)w_{\alpha}=\left(\nabla_{\hat{X}}\nabla_{\hat{Y}}v_{i}\right)_{V^{\bot}}\]
at $V=V_{0}$ by the definition of $\nabla_{X}^{\bot}Y^{v}$. Then, commuting $X$ and $Y$, we have proved (\ref{v2}). 
When 
\[X=\frac{\partial}{\partial x^{A}}\quad,\quad Y=\frac{\partial}{\partial a^{\alpha}_{i}}, \]
it is enough to prove
\begin{align}\label{v3}
\nabla^{\bot}_{X}Y^{v}=0\quad\mathrm{and}\quad \nabla^{\bot}_{Y}X^{v}=0
\end{align}
at $V_{0}$. 
First, we prove $\nabla^{\bot}_{X}Y^{v}=0$. 
Put $V(t,s):=\Gamma(te_{A},sE_{i}^{\alpha})$, where $E_{i}^{\alpha}$ is a matrix such that its $(j,\beta)$ component is $1$ if $(j,\beta)=(i,\alpha)$ and $0$ otherwise. 
Then we have 
\[\frac{\partial }{\partial t}\bigg|_{t=0}V(t,0)=X(V_{0})\quad\text{and}\quad \frac{\partial }{\partial s}\bigg|_{s=0}V(t,s)=Y(V(t,0)). \]
Since 
\[v_{j}(t,s):=v_{j}(\varphi^{-1}(te_{A}))+sE_{i}^{\alpha}(j,\beta)w_{\beta}(\varphi^{-1}(te_{A}))\] 
is a basis of $V(t,s)$, the vertical part of $Y$ at $V=V(t,0)$ is expressed as 
\[(Y^{v})(v_{j}(t,0))=\begin{cases}w_{\alpha}(te_{A})& \mathrm{if}\quad  j=i \\ \quad0 & \mathrm{if}\quad j\neq i\end{cases}\]
by (\ref{dec1}). 
By this expression and $\nabla_{\hat{X}}w_{\alpha}=0$ at $p_{0}$, we have proved 
\[\nabla^{\bot}_{X}Y^{v}=0. \]
Next, we prove $\nabla^{\bot}_{Y}X^{v}=0$. 
Put $V(t,s):=\Gamma(se_{A},tE_{i}^{\alpha})$. 
Then we have 
\[\frac{\partial }{\partial t}\bigg|_{t=0}V(t,0)=Y(V_{0})\quad\text{and}\quad \frac{\partial }{\partial s}\bigg|_{s=0}V(t,s)=X(V(t,0)). \]
Since  
\[v_{j}(t,s):=v_{j}(\varphi^{-1}(se_{A}))+tE_{i}^{\alpha}(j,\beta)w_{\beta}(\varphi^{-1}(se_{A}))\]
is a basis of $V(t,s)$, 
we have $X^{v}=0$ at each point $V=V(t,0)$ by (\ref{dec1}) and the fact that $\nabla_{\hat{X}}v_{i}=\nabla_{\hat{X}}w_{\alpha}=0$ at $p_{0}$. 
Thus, it is clear that 
\[\nabla^{\bot}_{Y}X^{v}=0. \]
Hence, the proof of torsion-free is completed, and this shows that $\tilde{\nabla}$ is the Levi--Civita connection of $(G_{m}(TN),\tilde{g})$. 
\end{proof}

By the expression (\ref{LC2}), we see that the horizontal part of $\tilde{\nabla}_{X}Y$ is 
\[\biggl[\nabla_{X}\hat{Y}+\frac{1}{2}k(R^{\bot}(\hat{X},\bullet),Y^{v})^{\flat}+\frac{1}{2}k(R^{\bot}(\hat{Y},\bullet),X^{v})^{\flat}\biggr]^{h}\]
and the vertical part of $\tilde{\nabla}_{X}Y$ is 
\[-\frac{1}{2}R^{\bot}(\hat{X},\hat{Y})+\nabla^{\bot}_{X}Y^{v}. \]

\begin{remark}
We can include a positive scaling constant $\alpha$ in the definition of Sasaki metric on $G_{m}(TN)$ as 
\[\tilde{g}_{\alpha}(X,Y):=g(\hat{X},\hat{Y})+\alpha k(X^{v},Y^{v}). \]
By slight modification of the proof of Proposition \ref{LC1}, one can easily see that the Levi--Citvita connection with respect to this Sasaki metric $\tilde{g}_{\alpha}$ is given by 
\begin{align*}
\tilde{\nabla}_{X}Y:=&\biggl[\nabla_{X}\hat{Y}+\frac{\alpha}{2}k(R^{\bot}(\hat{X},\bullet),Y^{v})^{\flat}+\frac{\alpha}{2}k(R^{\bot}(\hat{Y},\bullet),X^{v})^{\flat}\biggr]^{h}\\
&-\frac{1}{2}R^{\bot}(\hat{X},\hat{Y})+\nabla^{\bot}_{X}Y^{v}. 
\end{align*}
\end{remark}

\section{The tension field of a Gauss map}\label{TFG}
In this section, we calculate the tension field of the Gauss map associated with an immersion. 
It is also calculated by Jensen and Rigoli \cite{JensenRigoli} however, we give it for reader's convenience and to write it with our notations. 
Let $M$ be an $\ell$-dimensional manifold, $(N,g)$ be an $n$-dimensional Riemannian manifold and $F:M\to N$ be an immersion. 
Put $m:=n-\ell$ and consider this as the codimension. 
Recall, in Definition \ref{Gauss00}, a smooth map $\gamma_{F}:M\to G_{m}(TN)$ defined by
\[\gamma_{F}(p):=(F_{*}(T_{p}M))^{\bot}\subset T_{F(p)}N\]
is called the Gauss map associated with $F$. 

From now on, we calculate the tension field of the Gauss map $\gamma_{F}:M\to G_{m}(TN)$ 
with respect to the induced Riemannian metric $F^{*}g$ on $M$ and the Sasaki metric $\tilde{g}$ on $G_{m}(TN)$. 
First, we consider the first derivative of $\gamma_{F}$. 
It is clear that $(\pi\circ \gamma_{F})_{*}(X)=F_{*}(X)$ for all tangent vectors on $M$. 
Hence, the horizontal part of $(\gamma_{F})_{*}(X)$ is the horizontal lift of $F_{*}(X)$. 
On the other hand, it is well-known that the vertical part of $(\gamma_{F})_{*}(X)$ is given by the metric dual of the second fundamental form of $F$. 
Precisely, we have 
\[((\gamma_{F})_{*}(X))^{v}=-A(X,\ast)^{\flat\sharp}, \]
where $A$ is the second fundamental form of $F:M\to (N,g)$. 
Note that $A(X,\ast)$ is a section of $T^{*}M\otimes (F_{*}(TM))^{\bot}$. 
We have the identification of $T^{*}M$ with $F_{*}(TM)$ given by composition of the metric dual $T^{*}M\cong TM$ with respect to $F^{*}g$ and the push forward $F_{*}:TM \to F_{*}(TM)$. 
We also have the identification of $(F_{*}(TM))^{\bot}$ with $((F_{*}(TM))^{\bot})^{*}$ by a metric $g$. 
Hence, we have a bundle isomorphism $\iota:T^{*}M\otimes (F_{*}(TM))^{\bot}\to F_{*}(TM)\otimes ((F_{*}(TM))^{\bot})^{*}$, 
and denote $\iota(A(X,\ast))$ by $A(X,\ast)^{\flat\sharp}$, that is, the section of $F_{*}(TM)\otimes ((F_{*}(TM))^{\bot})^{*}$. 
Thus, we have
\begin{align}\label{dgamma}
(\gamma_{F})_{*}(X)=[F_{*}(X)]^{h}-A(X,\ast)^{\flat\sharp}. 
\end{align}
Based on the above formula, we have the following expression of the tension field of $\gamma_{F}:(M,F^{*}g)\to (G_{m}(TN),\tilde{g})$. 
Before the statement, we fix our notations. 
We denote the mean curvature vector field of $F:M\to (N,g)$ by $H$, the Riemannian curvature tensor of $(N,g)$ by $R$ and the normal connection by $\nabla^{N}$. 
Furthermore, we fix an orthonormal local frame $\{\,e_{i}\,\}_{i=1,\dots,\ell}$ of $TM$ with respect to $F^{*}g$ 
and an orthonormal local frame $\{\,\nu_{j}\,\}_{j=1,\dots,m}$ of $(F_{*}(TM))^{\bot}$ with respect to $g$, 
and denote $F_{*}(e_{k})$ by $\overline{e}_{k}$ for short. 

\begin{proposition}\label{tensofGauss}
The tension field of the Gauss map $\gamma_{F}:M \to G_{m}(TN)$ associated with $F:M\to (N,g)$ with respect to the induced metric $F^{*}g$ on $M$ and the Sasaki metric $\tilde{g}$ is given by
\begin{align*}
\tau(\gamma_{F})=&\left[H+\sum_{i=1}^{\ell}\sum_{k=1}^{\ell}R(A(e_{k},e_{i}),\overline{e}_{k},\overline{e}_{i},\bullet)^{\flat}\right]^{h}\\
&-(\nabla^{N} H)^{\flat\sharp}+\sum_{j=1}^{m}\sum_{k=1}^{\ell}\sum_{i=1}^{\ell}\nu_{j}^{*}\otimes R(\overline{e}_{i},\nu_{j},\overline{e}_{i},\overline{e}_{k})\overline{e}_{k}. 
\end{align*}
\end{proposition}
\begin{proof}
Fix a point $p$ in $M$ arbitrary. 
We can assume that the covariant derivative of $e_{i}$ with respect to the Levi--Civita connection on $(M,F^{*}g)$ vanishes at $p$ 
and the covariant derivative of $\nu_{j}$ with respect to the normal connection of $F$ also vanish at $p$. 
Then, the tension field of $\gamma_{F}$ at $p$ is given by
\[\tau(\gamma_{F})=\sum_{i=1}^{\ell}(\tilde{\nabla}d\gamma_{F})(e_{i},e_{i})
=\sum_{i=1}^{\ell}\tilde{\nabla}_{(\gamma_{F})_{*}(e_{i})}((\gamma_{F})_{*}(e_{i})). \]
By (\ref{dgamma}), we have 
\begin{align}\label{dgamma2}
\widehat{(\gamma_{F})_{*}(e_{i})}=F_{*}(e_{i})=\overline{e}_{i}\quad\mathrm{and}\quad ((\gamma_{F})_{*}(e_{i}))^{v}=-A(e_{i},\ast)^{\flat\sharp}. 
\end{align}
Substituting (\ref{dgamma2}) into (\ref{LC2}), we have
\begin{align*}
\tilde{\nabla}_{(\gamma_{F})_{*}(e_{i})}((\gamma_{F})_{*}(e_{i}))=&\biggl[\nabla_{\overline{e}_{i}}\overline{e}_{i}
-k(R^{\bot}(\overline{e}_{i},\bullet),A(e_{i},\ast)^{\flat\sharp})^{\flat}\biggr]^{h}\\
&-\nabla^{\bot}_{(\gamma_{F})_{*}(e_{i})}(A(e_{i},\ast)^{\flat\sharp}). 
\end{align*}
First, it is clear that $\nabla_{\overline{e}_{i}}\overline{e}_{i}=A(e_{i},e_{i})$ at $p$, and the sum of these from $i=1$ to $\ell$ is $H$. 
Next, since 
\begin{align*}
R^{\bot}(\overline{e}_{i},\bullet)=&\sum_{j=1}^{m}\nu_{j}^{*}\otimes \left(\sum_{k=1}^{\ell}g(R(\overline{e}_{i},\bullet)\nu_{j},\overline{e}_{k})\overline{e}_{k}\right), \\
A(e_{i},\ast)^{\flat\sharp}=&\sum_{j=1}^{m}\nu_{j}^{*}\otimes \left(\sum_{k=1}^{\ell}g(\nu_{j},A(e_{i},e_{k}))\overline{e}_{k}\right), 
\end{align*}
we have 
\begin{align*}
k(R^{\bot}(\overline{e}_{i},\bullet),A(e_{i},\ast)^{\flat\sharp})=&\sum_{j=1}^{m}\sum_{k=1}^{\ell}g(R(\overline{e}_{i},\bullet)\nu_{j},\overline{e}_{k})g(\nu_{j},A(e_{i},e_{k}))\\
=&-\sum_{k=1}^{\ell}R(A(e_{k},e_{i}),\overline{e}_{k},\overline{e}_{i},\bullet). 
\end{align*}
Finally, we calculate $\nabla^{\bot}_{(\gamma_{F})_{*}(e_{i})}(A(e_{i},\ast)^{\flat\sharp})$. 
By the definition (\ref{normalnabla}) of $\nabla^{\bot}$, we have
\begin{align*}
(\nabla^{\bot}_{(\gamma_{F})_{*}(e_{i})}(A(e_{i},\ast)^{\flat\sharp}))(\nu_{j})=&\left(\nabla_{e_{i}}\left(\sum_{k=1}^{\ell}g(\nu_{j},A(e_{i},e_{k}))\overline{e}_{k}\right)\right)_{V^{\bot}}. 
\end{align*}
Here we remark that $V=T^{\bot}_{p}M$ in this setting and hence $V^{\bot}=F_{*}(T_{p}M)$. 
Thus, $v_{V^{\bot}}$ is the tangential part of $v$ actually. 
At $p$, we have 
\[\left(\nabla_{e_{i}}\left(g(\nu_{j},A(e_{i},e_{k}))\overline{e}_{k}\right)\right)_{V^{\bot}}=g(\nu_{j},(\nabla_{e_{i}}A)(e_{i},e_{k}))\overline{e}_{k}. \]
By the Cadazzi equation, we have 
\[g(\nu_{j},(\nabla_{e_{i}}A)(e_{k},e_{i}))=g(\nu_{j},(\nabla_{e_{k}}A)(e_{i},e_{i}))+g(\nu_{j},R(\overline{e}_{i},\overline{e}_{k})\overline{e}_{i}). \]
Hence, we have
\begin{align*}
&\sum_{i=1}^{\ell}(\nabla^{\bot}_{(\gamma_{F})_{*}(e_{i})}(A(e_{i},\ast)^{\flat\sharp}))\\
=&(\nabla^{N} H)^{\flat\sharp}-\sum_{j=1}^{m}\sum_{k=1}^{\ell}\sum_{i=1}^{\ell}\nu_{j}^{*}\otimes R(\overline{e}_{i},\nu_{j},\overline{e}_{i},\overline{e}_{k})\overline{e}_{k}. 
\end{align*}
Here $(\nabla^{N} H)^{\flat\sharp}$ is an element of $\mathop{\mathrm{Hom}}((F_{*}(TM))^{\bot},F_{*}(TM))$ defined as the metric dual of the assignment $e_{i}\mapsto \nabla_{e_{i}}^{N} H$. 
We completed the proof. 
\end{proof}

\begin{remark}
For a Sasaki metric $\tilde{g}_{\alpha}$ including a positive scaling constant $\alpha$, 
the tension field of the Gauss map $\gamma_{F}:(M,F^{*}g) \to (G_{m}(TN),\tilde{g}_{\alpha})$ associated with $F:M\to (N,g)$ is given by
\begin{align*}
\tau(\gamma_{F})=&\left[H+\alpha\sum_{i=1}^{\ell}\sum_{k=1}^{\ell}R(A(e_{k},e_{i}),\overline{e}_{k},\overline{e}_{i},\bullet)^{\flat}\right]^{h}\\
&-(\nabla^{N} H)^{\flat\sharp}+\sum_{j=1}^{m}\sum_{k=1}^{\ell}\sum_{i=1}^{\ell}\nu_{j}^{*}\otimes R(\overline{e}_{i},\nu_{j},\overline{e}_{i},\overline{e}_{k})\overline{e}_{k}. 
\end{align*}
\end{remark}

\section{Variational vector fields of Gauss maps}\label{VVG}
In this section, we calculate the variational vector field of the 1-parameter family of Gauss maps associated with a 1-parameter family of immersions into a 1-parameter family of Riemannian manifolds. 
Let $N$ be an $n$-dimensional manifold and $g_{t}$ be a 1-parameter family of Riemannian metrics on $N$. 
Let $M$ be an $\ell$-dimensional manifold and $F_{t}:M\to N$ be a 1-parameter family of immersions. 
We assume that $t$ belongs to a time interval $(a,b)$. Define a map $F:M\times (a,b)\to N$ by  $F(p,t):=F_{t}(p)$. 
We denote the variational vector field of $F$ by 
\[V_{t}:=\frac{\partial F_{t}}{\partial t}. \]
Put $m:=n-\ell$ and consider this as the codimension. 
We put time dependent 2-tensors on $M$ and $N$ by 
\[P_{t}:=\frac{\partial}{\partial t} (F_{t}^{*}g_{t})\quad\mathrm{and}\quad Q_{t}:=\frac{\partial}{\partial t}g_{t}, \]
respectively. Then, for each time $t$, we get the Gauss map
\[\gamma_{F_{t}}:M\to G_{m}(TN)\quad;\quad \gamma_{F_{t}}(p):=(F_{t*}(T_{p}M))^{\bot_{t}}. \]
We remark that $\bot_{t}$ takes the normal part with respect to the ambient metric $g_{t}$, hence it depends on time. 
We will calculate the variational vector field of $\gamma_{F_{t}}$ at a fixed time $t=t_{0}$. 
A main tool to do this simply is Uhlenbeck's trick, which takes a nice time dependent orthonormal frame. 

First, fix a point $p$ in $M$ and an orthonormal local frames $(U,(f_{1},\dots,f_{\ell}))$ of $TM$ around $p$, 
with respect to the induced metric $F_{t_{0}}^{*}g_{t_{0}}$, such that its covariant derivative with respect to the Levi--Civita connection $F_{t_{0}}^{*}g_{t_{0}}$ vanishes at $p$. 
This $(f_{1},\dots,f_{\ell})$ does not depend on time. 
Next, solve the ODE: 
\begin{align}\label{Uh1}
\frac{\partial}{\partial t}e_{i}=-\frac{1}{2}\left(P_{t}(e_{i},\ast)\right)^{\flat_{t}}
\end{align}
with condition $e_{i}(t_{0})=f_{i}$. 
Here we put some remarks. 
For a time dependent vector field $e_{i}=e_{i}^{j}(x,t)\frac{\partial}{\partial x^{j}}$ on $M$, we define 
\[\frac{\partial}{\partial t}e_{i}:=\left(\frac{\partial}{\partial t}e_{i}^{j}(x,t)\right)\frac{\partial}{\partial x^{j}}, \]
and this is of course a time dependent vector field on $M$. 
In (\ref{Uh1}), $\flat_{t}$ means the metric dual of a 1-form on $M$ with respect to $F_{t}^{*}g_{t}$. 
Hence, (\ref{Uh1}) is a linear ODE for $e_{i}$ and we have a unique solution $e_{i}$ on $U$. 
By direct computation with (\ref{Uh1}), one can easily check that $\frac{\partial}{\partial t}((F_{t}^{*}g_{t})(e_{i},e_{j}))=0$. 
Hence $(U,(e_{1},\dots,e_{\ell}))$ is orthonormal with respect to $F_{t}^{*}g_{t}$ for all time $t$. 
This is a nice time dependent local frame field of $TM$. 

Next, we take a nice time dependent local frame field of 
\[(F_{t*}(TM))^{\bot_{t}}\]
by a similar way as above. 
For its purpose, we first define a connection $\nabla^{F}$ of the induced bundle $F^{*}(TN)$ over $M\times(a,b)$ as follows. 
For our aim, it is sufficient to define $\nabla_{\frac{\partial}{\partial t}}^{F}X$ for a section $X$ of the form $X=X^{\alpha}(x,t)\frac{\partial}{\partial y^{\alpha}}$. 
Then we define it by 
\begin{align}\label{tderi}
\nabla_{\frac{\partial}{\partial t}}^{F}X:=\left(\frac{\partial}{\partial t}X^{\alpha}+\Gamma^{\alpha}_{\gamma\delta}V^{\gamma}X^{\delta}  \right)\frac{\partial}{\partial y^{\alpha}}, 
\end{align}
where $(y^{\alpha})_{\alpha=1}^{n}$ is a local coordinates on $N$ and $\Gamma^{\alpha}_{\gamma\delta}$ are Christoffel symbols of $g_{t}$ with respect to $(y^{\alpha})_{\alpha=1}^{n}$. 
Actually, this is just the induced connection on $(a,b)$ by a map $(a,b)\ni t\mapsto F_{t}(x)\in N$ for each $x$ in $M$. 
Then, one can check that this does not depend on the choice of coordinates. 
Let $(U,(\xi_{1},\dots,\xi_{m}))$ be a time independent orthonormal local frame of $(F_{t_{0}*}(TM))^{\bot_{t_{0}}}$ around $p$, 
with respect to $g_{t_{0}}$, such that its covariant derivative with respect to the normal connection of $F_{t_{0}}$ vanishes at $p$. 
Next, solve the ODE: 
\begin{align}\label{Uh2}
\nabla_{\frac{\partial}{\partial t}}^{F}\nu_{i}
=-\frac{1}{2}\left(\left(Q_{t}(\nu_{i},\bullet)\right)^{\flat_{t}}\right)^{\bot_{t}}
-Q_{t}(\nu_{i},\overline{e}_{k})\overline{e}_{k}
-g_{t}\left(\nu_{i},\nabla_{\frac{\partial}{\partial t}}^{F}\overline{e}_{k}\right)\overline{e}_{k}
\end{align}
with condition $\nu_{i}(t_{0})=\xi_{i}$. 
Here $\flat_{t}$ is the metric dual of a 1-form on $N$ with respect to $g_{t}$, 
$(U,(e_{1},\dots,e_{\ell}))$ is the nice time dependent local frame field taken as above and we put $\overline{e}_{j}:=F_{t*}(e_{j})$. 
Since (\ref{Uh2}) is well-defined and a linear ODE for $\nu_{i}\in \Gamma(U,F^{*}(TM)|_{U})$ of the from $\nu_{i}=\nu_{i}^{\alpha}(x,t)\frac{\partial}{\partial y^{\alpha}}$, 
we have a unique solution $\nu_{i}$ on $U$. 
Additionally, the time derivative (\ref{tderi}) satisfies the following Leibniz rule: 
\begin{align}\label{Lei1}
\frac{\partial}{\partial t}(g_{t}(X,Y))=Q_{t}(X,Y)+g_{t}\left(\nabla^{F}_{\frac{\partial}{\partial t}}X,Y \right)+g_{t}\left(X,\nabla^{F}_{\frac{\partial}{\partial t}}Y\right). 
\end{align}
Then, by (\ref{Uh2}) and (\ref{Lei1}), we have $\frac{\partial}{\partial t}(g_{t}(\nu_{i},\overline{e}_{j}))=0$. 
Hence, we have proved that $\nu_{i}$ is a section of $(F_{t}^{*}(TM))^{\bot_{t}}$ actually. 
Again, by (\ref{Uh2}) and (\ref{Lei1}) with the fact that $g_{t}(\nu_{i},\overline{e}_{k})=0$, we have $\frac{\partial}{\partial t}(g_{t}(\nu_{i},\nu_{j}))=0$. 
Hence, we have proved that $(U,(\nu_{1},\dots,\nu_{m}))$ is an orthonormal local frame field of $(F_{t}^{*}(TM))^{\bot_{t}}$ with respect to $g_{t}$ for each $t\in (a,b)$. 
By using these nice time dependent orthonormal frame fields, we prove the following. 

\begin{proposition}\label{varGauss}
The variational vector field of $\gamma_{F_{t}}:M\to G_{m}(TN)$ is given by
\[\frac{\partial}{\partial t}\gamma_{F_{t}}=\left[V_{t}\right]^{h}-\left(\nabla^{N}V_{t}\right)^{\flat\sharp}
-\sum_{i=1}^{m}\sum_{k=1}^{\ell}\nu_{i}^{*}\otimes\left(Q_{t}(\nu_{i},\overline{e}_{k})\right)\overline{e}_{k}. \]
\end{proposition}
\begin{proof}
Since $\pi(\gamma_{F_{t}}(p))=F_{t}(p)$, it is clear that 
\[\pi_{*}\left(\frac{\partial}{\partial t}\gamma_{F_{t}}\right)=\frac{\partial F_{t}}{\partial t}=V_{t}. \]
Hence the horizontal part of $\frac{\partial}{\partial t}\gamma_{F_{t}}$ is $[V_{t}]^{h}$. 
Let $e_{i}$ and $\nu_{j}$ be as above. Then, at $p$, we have
\[\gamma_{F_{t}}(p)=(F_{t*}(T_{p}M))^{\bot_{t}}=\mathop{\mathrm{Span}}\{\nu_{1}(p,t),\dots,\nu_{m}(p,t)\}. \]
Then, by (\ref{dec1}), we have
\[\left(\frac{\partial}{\partial t}\gamma_{F_{t}}\right)^{v}=\sum_{i=1}^{m}\nu_{i}^{*}\otimes \left(\nabla_{\frac{\partial}{\partial t}}^{F}\nu_{i}\right)^{\top_{t}}. \]
Since $\nu_{i}$ satisfies (\ref{Uh2}), we have
\[\left(\nabla_{\frac{\partial}{\partial t}}^{F}\nu_{i}\right)^{\top_{t}}=
-Q_{t}(\nu_{i},\overline{e}_{k})\overline{e}_{k}
-g_{t}\left(\nu_{i},\nabla_{\frac{\partial}{\partial t}}^{F}\overline{e}_{k}\right)\overline{e}_{k}. \]
By the definition (\ref{tderi}), we can easily see that
\begin{align}\label{tanderi}
\nabla_{\frac{\partial}{\partial t}}^{F}\overline{e}_{k}=\nabla_{\frac{\partial}{\partial t}}^{F}(F_{t*}(e_{k}))=&\nabla_{e_{k}}V_{t}+F_{t*}\left(\frac{\partial}{\partial t}e_{k}\right). 
\end{align}
Since $\nu_{i}$ is a section of $(F_{t}^{*}(TM))^{\bot_{t}}$, the inner product of $\nu_{i}$ and the second term of the right hand side of (\ref{tanderi}) is zero. 
Thus, we have 
\[\left(\frac{\partial}{\partial t}\gamma_{F_{t}}\right)^{v}=-\sum_{i=1}^{m}\nu_{i}^{*}\otimes \left(\sum_{k=1}^{\ell}g_{t}(\nu_{i},\nabla^{N}_{e_{k}}V_{t})\overline{e}_{k}\right)
-\sum_{i=1}^{m}\sum_{k=1}^{\ell}\nu_{i}^{*}\otimes\left(Q_{t}(\nu_{i},\overline{e}_{k})\right)\overline{e}_{k}. \]
This completes the proof. 
\end{proof}
\section{Gauss maps of the coupled flow}\label{GCF}
This section is devoted to the proof of Theorem \ref{main1}. 
First, we recall the assumptions. 
Let $(N,g_{t})$, $f_{t}:N\to \mathbb{R}$ and $F_{t}:M\to N$ be a 1-parameter family of $n$-dimensional Riemannian manifolds, 
smooth functions on $N$ and immersions from an $\ell$-dimensional manifold $M$ respectively defined on a time interval $[0,T)$ satisfying (\ref{rf1}) and (\ref{rmcf1}). 

\begin{proof}[Proof of Theorem \ref{main1}]
Fix a point $p$ in $M$. We prove the equation at $p$. Let $\{\,e_{i}\,\}_{i=1}^{\ell}$ and $\{\,\nu_{j}\,\}_{j=1}^{m}$ be nice time dependent local frames around $p$ taken as in Section \ref{VVG}. 
At each fixed time $t\in(0,T)$, by taking the vertical part of the equation proved in Proposition \ref{tensofGauss}, we have
\begin{align}\label{eqC}
\begin{aligned}
\tau(\gamma_{F_{t}})^{v}=&-(\nabla^{N} H)^{\flat\sharp}+\sum_{j=1}^{m}\sum_{k=1}^{\ell}\sum_{i=1}^{\ell}\nu_{j}^{*}\otimes R(\bar{e}_{i},\nu_{j},\bar{e}_{i},\bar{e}_{k})\bar{e}_{k}\\
=&-(\nabla^{N} H)^{\flat\sharp}+\sum_{j=1}^{m}\sum_{k=1}^{\ell}\nu_{j}^{*}\otimes \mathrm{Ric}(\nu_{j},\bar{e}_{k})\bar{e}_{k}\\
&-\sum_{j=1}^{m}\sum_{k=1}^{\ell}\sum_{i=1}^{m}\nu_{j}^{*}\otimes R(\nu_{i},\nu_{j},\nu_{i},\bar{e}_{k})\bar{e}_{k}, 
\end{aligned}
\end{align}
where we write $H_{g_{t}}(F_{t})$ as $H$ for short, and the Riemannian curvature tensor of $g_{t}$ as $R$. 
The third term of the right hand side is nothing but $-\mathcal{R}(g_{t})\circ \gamma_{F_{t}}$ since $\{\,\nu_{j}\,\}_{j=1}^{m}$ is an orthonormal basis of $\gamma_{F_{t}}$ 
and $\{\,\bar{e}_{i}\,\}_{i=1}^{\ell}$ is an orthonormal basis of $(\gamma_{F_{t}})^{\bot}$, see Definition \ref{RgF}. 
Next, by taking the vertical part of the equation proved in Proposition \ref{varGauss}, we have
\begin{align}\label{varGauss2}
\left(\frac{\partial \gamma_{F_{t}}}{\partial t}\right)^{v}=-\left(\nabla^{N}V_{t}\right)^{\flat\sharp}-\sum_{j=1}^{m}\sum_{k=1}^{\ell}\nu_{j}^{*}\otimes\left(Q_{t}(\nu_{j},\overline{e}_{k})\right)\overline{e}_{k}. 
\end{align}
Since $g_{t}$ and $F_{t}$ satisfy (\ref{rf1}) and (\ref{rmcf1}), we have
\begin{align}\label{eqA}
V_{t}=\frac{\partial F_{t}}{\partial t}=H\quad\text{and}\text \quad Q_{t}=\frac{\partial g_{t}}{\partial t}=-\mathrm{Ric}+f_{t}g_{t}. 
\end{align}
Since $\nu_{i}$ and $\overline{e}_{k}$ are orthogonal, we have
\begin{align}\label{eqB}
Q_{t}(\nu_{i},\overline{e}_{k})=-\mathrm{Ric}(\nu_{i},\overline{e}_{k})+f_{t}g_{t}(\nu_{i},\overline{e}_{k})=-\mathrm{Ric}(\nu_{i},\overline{e}_{k}). 
\end{align}
Substituting (\ref{eqA}) and (\ref{eqB}) into (\ref{varGauss2}), we have 
\begin{align}\label{eqD}
\left(\frac{\partial \gamma_{F_{t}}}{\partial t}\right)^{v}=-\left(\nabla^{N}H\right)^{\flat\sharp}+\sum_{j=1}^{m}\sum_{k=1}^{\ell}\nu_{j}^{*}\otimes\left(\mathrm{Ric}(\nu_{j},\overline{e}_{k})\right)\overline{e}_{k}. 
\end{align}
Comparing (\ref{eqC}) with (\ref{eqD}), we have 
\begin{align}\label{eqE}
\left(\frac{\partial \gamma_{F_{t}}}{\partial t}\right)^{v}=\tau(\gamma_{F_{t}})^{v}+\mathcal{R}(g_{t})\circ \gamma_{F_{t}}. 
\end{align}
Then, Theorem \ref{main1} has been proved. 
\end{proof}
\section{Associated subsolutions}\label{Asub}
This section is devoted to the proof of Theorem \ref{main2}. 
First, we recall our situation. 
Let $(N,g_{t})$, $f_{t}:N\to \mathbb{R}$ and $F_{t}:M\to N$ be a 1-parameter family of $n$-dimensional Riemannian manifolds, 
smooth functions on $N$ and immersions from an $(n-1)$-dimensional manifold $M$ respectively defined on a time interval $[0,T)$ satisfying (\ref{rf1}) and (\ref{rmcf1}). 
Assume that a smooth function $\rho:\mathbb{P}(TN)\to \mathbb{R}$ satisfies (\ref{horicon}) and (\ref{hessC}). 

\begin{proof}[Proof of Theorem \ref{main2}]
In general, we have
\begin{align}
&\frac{\partial}{\partial t}\left(\rho\circ\gamma_{F_{t}}\right)=\tilde{g_{t}}\left(\nabla^{t}\rho,\frac{\partial \gamma_{F_{t}}}{\partial t} \right) \label{eq61}\\
&\Delta_{F_{t}^{*}g_{t}}\left(\rho\circ\gamma_{F_{t}}\right)=\mathop{\mathrm{tr}}\left( \gamma_{F_{t}}^{*}\left(\mathop{\mathrm{Hess}_{t}}\rho \right)\right) \label{eq62}
+\tilde{g_{t}}\left(\nabla^{t}\rho,\tau(\gamma_{F_{t}}) \right), 
\end{align}
where $\gamma_{F_{t}}^{*}\left(\mathop{\mathrm{Hess}_{t}}\rho \right)$ is the pull-back of the 2-tensor $\mathop{\mathrm{Hess}_{t}}\rho$ by $\gamma_{F_{t}}$. 
By (\ref{eq61}), (\ref{eq62}) and (\ref{horicon}), we have
\[\left(\frac{\partial}{\partial t}-\Delta\right)(\rho\circ\gamma_{F_{t}})
=\tilde{g_{t}}\left(\left(\nabla^{t}\rho\right)^{v},\left(\frac{\partial \gamma_{F_{t}}}{\partial t}-\tau(\gamma_{F_{t}})\right)^{v}\right)-\mathop{\mathrm{tr}}\left( \gamma_{F_{t}}^{*}\left(\mathop{\mathrm{Hess}_{t}}\rho \right)\right). \]
Since $\gamma_{F_{t}}$ is a vertically harmonic map heat flow by Corollary \ref{cor1}, we have
\[\left(\frac{\partial}{\partial t}-\Delta\right)(\rho\circ\gamma_{F_{t}})
=-\mathop{\mathrm{tr}}\left( \gamma_{F_{t}}^{*}\left(\mathop{\mathrm{Hess}_{t}}\rho \right)\right). \]
Combining the assumption (\ref{hessC}), we have
\[\left(\frac{\partial}{\partial t}-\Delta\right)(\rho\circ\gamma_{F_{t}})\leq C|(\gamma_{F_{t}})_{*}|^2. \]
By (\ref{dgamma2}), it is easy to show that $|(\gamma_{F_{t}})_{*}|^2=(n-1)+|A(F_{t})|^2$. 
Thus, the proof has been completed. 
\end{proof}
\appendix\section{\empty}
In this appendix, we put a problem. 
Let $(N,g)$ be an $n$-dimensional Riemannian manifold and $M$ be an $\ell$-dimensional manifold. Put $m:=n-\ell$. 
We denote by $C^{\infty}_{\mathrm{imm}}(M,G_{m}(TN))$ the set of all smooth maps $\gamma:M\to G_{m}(TN)$ so that $\pi\circ \gamma:M\to N$ becomes an immersion, 
where $\pi:G_{m}(TN)\to N$ is the projection. 
Then, one can consider the following evolution equation defined on $C^{\infty}_{\mathrm{imm}}(M,G_{m}(TN))$; 
\begin{align}\label{general2}
\left(\frac{\partial \gamma_{t}}{\partial t}\right)^{v}=\tau(\gamma_{t})^{v}+\mathcal{R}(g)\circ \gamma_{t}, 
\end{align}
where $\tau(\gamma_{t})$ is the tension field of $\gamma_{t}:M\to G_{m}(TN)$ 
with respect to a Riemannian metric $(\pi\circ \gamma_{t})^{*}g$ on $M$ and the Sasaki metric $\tilde{g}$ on $G_{m}(TN)$. 
The difference between (\ref{general2}) and (\ref{general}) is time-independence of $\mathcal{R}(g)$. 

Let $(N,g)$ be Ricci-flat. 
Then $g_{t}:=g$ is a solution of (\ref{rf1}) with $f_{t}:=0$. 
Hence, by Theorem \ref{main1}, Gauss maps $\gamma_{t}:=\gamma_{F_{t}}$ associated with a mean curvature flow $F_{t}:M\to (N,g)$ actually satisfy (\ref{general2}). 
Thus, when $(N,g)$ is Ricci-flat, (\ref{general2}) has many solutions and is meaningful. 
However, when the Ricci-flat condition is dropped, we do not know the evolution equation (\ref{general2}) is meaningful so far. 
Hence, the following might be an interesting problem. 

\begin{problem}
For a general Riemannian manifold $(N,g)$, does there exist a functional $E$ on $C^{\infty}_{\mathrm{imm}}(M,G_{m}(TN))$ such that 
its gradient flow equation is equal to the evolution equation (\ref{general2})?
\end{problem}

\end{document}